\documentclass[11pt,reqno]{amsart}
\usepackage{enumerate}
\usepackage{amsfonts}
\usepackage{amsmath}
\usepackage{amsthm}
\usepackage{amssymb}
\usepackage{latexsym}
\usepackage{multicol}
\usepackage{verbatim}
\usepackage{amscd}
\usepackage{hyperref}
\usepackage{amsmath, amscd}

\advance\textwidth by 1.2in \advance\oddsidemargin by -.6in \advance\evensidemargin by -.6in
\newtheorem*{cor}{Corollary}
\newtheorem*{lem}{Lemma}
\newtheorem*{prop}{Proposition}

\theoremstyle{definition}

\theoremstyle{definition}
\newtheorem*{thm}{Theorem}

\newtheorem*{rem}{Remark}

\newenvironment{pf}{\proof}{\endproof}
\newcounter{cnt}
\newenvironment{enumerit}{\begin{list}{{\hfill\rm(\roman{cnt})\hfill}}{%
\settowidth{\labelwidth}{{\rm(iv)}}\leftmargin=\labelwidth%
\advance\leftmargin by \labelsep\rightmargin=0pt\usecounter{cnt}}}{\end{list}} \makeatletter
\def\mydggeometry{\makeatletter\dg@YGRID=1\dg@XGRID=20\unitlength=0.003pt\makeatother}
\makeatother \theoremstyle{remark}

\numberwithin{equation}{section}

\let\bwdg\bigwedge
\def\bigwedge{{\textstyle\bwdg}}

\begin{document}

\newcommand{\thmref}[1]{Theorem~\ref{#1}}
\newcommand{\secref}[1]{Section~\ref{#1}}
\newcommand{\lemref}[1]{Lemma~\ref{#1}}
\newcommand{\propref}[1]{Proposition~\ref{#1}}
\newcommand{\corref}[1]{Corollary~\ref{#1}}
\newcommand{\remref}[1]{Remark~\ref{#1}}
\newcommand{\defref}[1]{Definition~\ref{#1}}
\newcommand{\er}[1]{(\ref{#1})}
\newcommand{\id}{\operatorname{id}}
\newcommand{\ord}{\operatorname{\emph{ord}}}
\newcommand{\sgn}{\operatorname{sgn}}
\newcommand{\wt}{\operatorname{wt}}
\newcommand{\tensor}{\otimes}
\newcommand{\from}{\leftarrow}
\newcommand{\nc}{\newcommand}
\newcommand{\rnc}{\renewcommand}
\newcommand{\dist}{\operatorname{dist}}
\newcommand{\qbinom}[2]{\genfrac[]{0pt}0{#1}{#2}}
\nc{\cal}{\mathcal} \nc{\goth}{\mathfrak} \rnc{\bold}{\mathbf}
\renewcommand{\frak}{\mathfrak}
\newcommand{\supp}{\operatorname{supp}}
\newcommand{\Irr}{\operatorname{Irr}}
\renewcommand{\Bbb}{\mathbb}
\nc\bomega{{\mbox{\boldmath $\omega$}}} \nc\bpsi{{\mbox{\boldmath $\Psi$}}}
 \nc\balpha{{\mbox{\boldmath $\alpha$}}}
 \nc\bpi{{\mbox{\boldmath $\pi$}}}
\nc\bsigma{{\mbox{\boldmath $\sigma$}}} \nc\bcN{{\mbox{\boldmath $\cal{N}$}}} \nc\bcm{{\mbox{\boldmath $\cal{M}$}}} \nc\bLambda{{\mbox{\boldmath
$\Lambda$}}}
\def\sq{\operatorname{sq}}
\newcommand{\lie}[1]{\mathfrak{#1}}
\makeatletter
\def\section{\def\@secnumfont{\mdseries}\@startsection{section}{1}%
  \z@{.7\linespacing\@plus\linespacing}{.5\linespacing}%
  {\normalfont\scshape\centering}}
\def\subsection{\def\@secnumfont{\bfseries}\@startsection{subsection}{2}%
  {\parindent}{.5\linespacing\@plus.7\linespacing}{-.5em}%
  {\normalfont\bfseries}}
\makeatother
\def\subl#1{\subsection{}\label{#1}}
 \nc{\Hom}{\operatorname{Hom}}
  \nc{\mode}{\operatorname{mod}}
\nc{\End}{\operatorname{End}} \nc{\wh}[1]{\widehat{#1}} \nc{\Ext}{\operatorname{Ext}} \nc{\ch}{\text{ch}} \nc{\ev}{\operatorname{ev}}
\nc{\Ob}{\operatorname{Ob}} \nc{\soc}{\operatorname{soc}} \nc{\rad}{\operatorname{rad}} \nc{\head}{\operatorname{head}}
\def\Im{\operatorname{Im}}
\def\gr{\operatorname{gr}}
\def\mult{\operatorname{mult}}
\def\Max{\operatorname{Max}}
\def\ann{\operatorname{Ann}}
\def\sym{\operatorname{sym}}
\def\Res{\operatorname{\br^\lambda_A}}
\def\und{\underline}
\def\Lietg{$A_k(\lie{g})(\bsigma,r)$}

 \nc{\Cal}{\cal} \nc{\Xp}[1]{X^+(#1)} \nc{\Xm}[1]{X^-(#1)}
\nc{\on}{\operatorname} \nc{\Z}{{\bold Z}} \nc{\J}{{\cal J}} \nc{\C}{{\bold C}} \nc{\Q}{{\bold Q}}
\renewcommand{\P}{{\cal P}}
\nc{\N}{{\Bbb N}} \nc\boa{\bold a} \nc\bob{\bold b} \nc\boc{\bold c} \nc\bod{\bold d} \nc\boe{\bold e} \nc\bof{\bold f} \nc\bog{\bold g}
\nc\boh{\bold h} \nc\boi{\bold i} \nc\boj{\bold j} \nc\bok{\bold k} \nc\bol{\bold l} \nc\bom{\bold m} \nc\bon{\bold n} \nc\boo{\bold o}
\nc\bop{\bold p} \nc\boq{\bold q} \nc\bor{\bold r} \nc\bos{\bold s} \nc\boT{\bold t} \nc\boF{\bold F} \nc\bou{\bold u} \nc\bov{\bold v}
\nc\bow{\bold w} \nc\boz{\bold z} \nc\boy{\bold y} \nc\ba{\bold A} \nc\bb{\bold B} \nc\bc{\bold C} \nc\bd{\bold D} \nc\be{\bold E} \nc\bg{\bold
G} \nc\bh{\bold H} \nc\bi{\bold I} \nc\bj{\bold J} \nc\bk{\bold K} \nc\bl{\bold L} \nc\bm{\bold M} \nc\bn{\bold N} \nc\bo{\bold O} \nc\bp{\bold
P} \nc\bq{\bold Q} \nc\br{\bold R} \nc\bs{\bold S} \nc\bt{\bold T} \nc\bu{\bold U} \nc\bv{\bold V} \nc\bw{\bold W} \nc\bz{\bold Z} \nc\bx{\bold
x} \nc\KR{\bold{KR}} \nc\rk{\bold{rk}} \nc\het{\text{ht }}

\nc\toa{\tilde a} \nc\tob{\tilde b} \nc\toc{\tilde c} \nc\tod{\tilde d} \nc\toe{\tilde e} \nc\tof{\tilde f} \nc\tog{\tilde g} \nc\toh{\tilde h}
\nc\toi{\tilde i} \nc\toj{\tilde j} \nc\tok{\tilde k} \nc\tol{\tilde l} \nc\tom{\tilde m} \nc\ton{\tilde n} \nc\too{\tilde o} \nc\toq{\tilde q}
\nc\tor{\tilde r} \nc\tos{\tilde s} \nc\toT{\tilde t} \nc\tou{\tilde u} \nc\tov{\tilde v} \nc\tow{\tilde w} \nc\toz{\tilde z}

\nc\parlambda{\lambda=(\lambda_1\ge\lambda_2\ge\cdots\ge\lambda_k>0)}
\title[]{Square partitions and Catalan numbers}
\author{Matthew Bennett, Vyjayanthi Chari, R. J. Dolbin and Nathan Manning}
\thanks{VC was partially supported by the NSF grant DMS-0901253}
\address{Department of Mathematics, University of
California, Riverside, CA 92521.} \email{mbenn002@math.ucr.edu}\email{chari@math.ucr.edu}
 \email{rjdolbin@math.ucr.edu}\email {nmanning@math.ucr.edu}

 \begin{abstract} For each  integer $k\ge 1$, we define an algorithm which associates to a partition  whose maximal value is at most $k$ a certain subset of all partitions. In the case when we begin with a partition $\lambda$ which is square, i.e $\lambda=\lambda_1\ge\cdots\ge\lambda_k>0$, and  $\lambda_1=k,\lambda_k=1$, then applying the algorithm $\ell$ times gives rise to a set whose cardinality is either the Catalan number $c_{\ell-k+1}$ (the self dual case) or twice the Catalan number. The algorithm defines a tree and we study the propagation of the tree, which is not in the isomorphism class of the usual Catalan tree. The algorithm can also be modified to produce a two--parameter family of sets and the resulting cardinalities of the sets are the ballot numbers.  Finally, we give a conjecture on the rank of a particular module for the ring of symmetric functions in $2\ell+m$ variables.
 \end{abstract}

\maketitle

\section*{Introduction} The Catalan numbers $c_\ell$, where $\ell$ is a non--negative integer, appear in a large number combinatorial settings and in \cite{Stanley} one can find sixty--six interpretations of the  Catalan numbers. Many of these  generalize to the ballot numbers $b_{\ell, m}$ where $\ell, m$ are both non--negative integers and $c_\ell=b_{\ell,0}$. These numbers also appear in the representation theory of the Lie algebra $\lie {sl_2}$ in the following way. Consider the $(2\ell+m)$--fold tensor product of the natural representation of $\lie{sl_2}$. As a representation of $\lie{sl_2}$ this tensor product is completely reducible and the multiplicity of the $(m+1)$--dimensional irreducible representation of $\lie{sl_2}$ is $b_{\ell,m}$.

The current  paper was motivated by the study of the category of  finite--dimensional representations of the affine Lie algebra associated to $\lie{sl_2}$ and an attempt begun in \cite{CG1}, \cite{CG2} to develop a theory of highest weight categories after \cite{CPS}. In the course of their work, Chari and Greenstein realized that one of the  results required for this would be  to prove that a  certain naturally defined  module for the ring of symmetric functions in $2\ell$ variables is free of rank equal to the Catalan number. In fact, it has  turned out that finer results are needed, namely one would need the basis of the free module and also an extension to more general modules for the ring of symmetric functions.
The conjecture is made precise in Section 4 of this paper.

In section one of this paper, we define an algorithm   which when applied $\ell$ times to a  partition $\lambda=k\ge\lambda_2\cdots\ge\lambda_{k-1}>1>0$ gives a subset of partitions with cardinality equal to the Catalan number $c_{\ell-k+1}$.  In fact we prove that  this algorithm defines an equivalence relation on the set $\cal P^\ell$  of all partitions $\mu=\mu_1\ge\cdots\ge \mu_\ell>0$ which satisfy $\mu_1\le \ell$. The algorithm defines an ordered rooted tree which is labeled either by pairs of non--negative integers or by single non--negative integers, and thus is very different from the usual Catalan tree. In addition, our algorithm uses a certain involution $\tau_\ell$ which defines a duality on $\cal P^\ell$. Our proofs in Section 2 are algebraic rather than combinatorial. A reason for this is that we were unable to find  any natural bijection between the sets we  describe and the usual sets giving rise to the Catalan numbers which keeps track of the  duality.

In Section 3 we describe a generalization of the algorithm which gives rise to ballot numbers $b_{\ell,m}$. This time, the algorithm describes a set of $m$ rooted ordered  trees. To do this, we  prove an alternating identity for the ballot numbers, which generalizes the well--known one \cite{Aigner} for Catalan numbers. In section 4, we also discuss further directions in which these algorithms could be generalized.

\vskip 24pt

{\em Acknowledgements:  We are grateful to Jacob Greenstein for discussions and his help with a Mathematica program which was crucial in the early stages of this work.}

\section{The Main Results}

\subsection{} Throughout the paper $\bn$ denotes the set of natural numbers and $\bz_+$ the set of non--negative integers. By a partition $\lambda$ with $n$ parts, we  mean a decreasing sequence  $$\lambda=(\lambda_1\ge\lambda_2\ge\cdots\ge\lambda_n>0)$$ of positive integers. We denote the set of all partitions by $\cal P$. Given $\lambda\in \cal P$ set  $$\lambda\setminus\{\lambda_n\}= \lambda_1\ge\cdots\ge\lambda_{n-1},$$
and for $0 < \lambda_{n+1} \leq \lambda_n$ set $$(\lambda:\lambda_{n+1})=(\lambda_1\ge\lambda_2\ge\cdots\ge\lambda_n\ge\lambda_{n+1}>0).$$
For $\ell,m\ge 0$, let $b_{\ell,m}$ be the ballot numbers given by $$b_{\ell,m}=\binom{2\ell+m}{\ell}-\binom{2\ell+m}{\ell-1},$$ and set $c_\ell=b_{\ell,0}.$

 \subsection{} For $k, n \in\bn$, let $\cal P^{n,k}$ be the set of partitions with exactly  $n$ parts where no part is bigger than $k$, i.e.
$$\cal P^{n,k}=\{\lambda=(\lambda_1\ge\lambda_2\ge\cdots\ge\lambda_n>0):\lambda_1\le k\}.$$
 Let $\tau_k:\cal P^{n,k}\to\cal P^{n,k}$ be defined by $$\tau_k(\lambda_1\ge\cdots\ge \lambda_n)=k+1-\lambda_n\ge\cdots\ge k+1-\lambda_1.$$ Clearly $\tau_k$ is a bijection of order two.
To understand the map $\tau_k$ in terms of Young diagrams, it is convenient to think of $\cal P^{n,k}$ as  the set of partitions whose Young diagrams lie in an $n\times k$ rectangle and with exactly $n$ rows. The Young diagram of $\tau_k(\lambda)$ is obtained by taking the skew diagram $(n\times (k+1))\setminus\lambda$  and rotating it by one hundred eighty degrees.

   As an example, we can regard the partition $\lambda = (7\ge 6\ge 5\ge 5\ge 4\ge 2)$ as an element of  $\cal P^{6,9}$ in which case we have  $\tau_9(\lambda) =( 8\ge 6\ge 5\ge 5\ge 4\ge  3)$,  and pictorially, we get

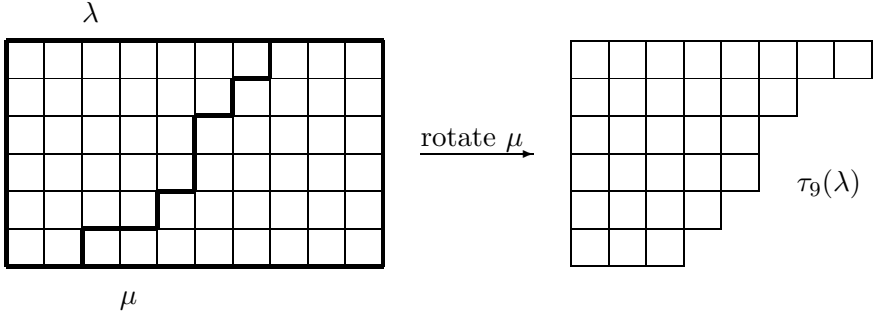
\setlength{\unitlength}{5mm}
\begin{picture}(10,10)
 \linethickness{.5mm}

 \put(1,8){\line(1,0){10}}
 \put(1,2){\line(1,0){10}}
 \put(1,8){\line(0,-1){6}}
 \put(11,8){\line(0,-1){6}}
 \put(3,2){\line(0,1){1}}
 \put(3,3){\line(1,0){2}}
 \put(5,3){\line(0,1){1}}
 \put(5,4){\line(1,0){1}}
 \put(6,4){\line(0,1){2}}
 \put(6,6){\line(1,0){1}}
 \put(7,6){\line(0,1){1}}
 \put(7,7){\line(1,0){1}}
 \put(8,7){\line(0,1){1}}

 \linethickness{.1mm}

 \multiput(1,3)(0,1){5}{\line(1,0){10}}
 \multiput(2,2)(1,0){9}{\line(0,1){6}}

 \put(3,8.5){$\lambda$}
 \put(4,1){$\mu$}

 \put(12,5){\vector(1,0){3}}
 \put(12,5.2){rotate $\mu$ }

 \put(16,8){\line(0,-1){6}}
 \put(17,8){\line(0,-1){6}}
 \put(18,8){\line(0,-1){6}}
 \put(19,8){\line(0,-1){6}}
 \put(20,8){\line(0,-1){5}}
 \put(21,8){\line(0,-1){4}}
 \put(22,8){\line(0,-1){2}}
 \put(23,8){\line(0,-1){1}}
 \put(24,8){\line(0,-1){1}}

 \put(16,2){\line(1,0){3}}
 \put(16,3){\line(1,0){4}}
 \put(16,4){\line(1,0){5}}
 \put(16,5){\line(1,0){5}}
 \put(16,6){\line(1,0){6}}
 \put(16,7){\line(1,0){8}}
 \put(16,8){\line(1,0){8}}

 \put(22,4){$\tau_9 ( \lambda )$}

\end{picture}

\subsection{}Set $\cal P^k=\cal P^{k,k}$.  Given $\lambda\in\cal P^k$ and $\ell,k\in\bn$ with $\ell\ge k$, define subsets $ \cal P^\ell(\lambda)$ of $\cal P^\ell$ inductively, by
\begin{gather*}\cal P^k(\lambda)=\{\lambda\}\cup\{\tau_k\lambda\},\ \ \   \cal P^\ell(\lambda)=\cal P^\ell_{\rm d}(\lambda)\cup\cal P^\ell_\tau(\lambda),\end{gather*} where
\begin{gather*}\cal P_{\rm d}^\ell(\lambda)=\{\mu\in\cal P^\ell:   \mu\setminus\{\mu_\ell\}\in\cal P^{\ell-1}(\lambda)\},\\ \\ \cal P^\ell_\tau(\lambda)= \{\mu\in\cal P^\ell:  \tau_\ell\mu\setminus\{\ell+1-\mu_1\}\in \  \cal P^{\ell-1}(\lambda)\}=\tau_\ell\cal P^\ell_{\rm d}(\lambda).\end{gather*}
 Clearly \begin{equation*}\label{taustab}\cal P^\ell(\tau_k\lambda)= \tau_\ell\cal P^\ell(\lambda)=\cal P^\ell(\lambda).\end{equation*} \begin{lem} \label{min} Let $\mu\in\cal P^\ell$. Then
\begin{gather*} \mu\in\cal P^\ell_{\rm d}(\lambda)\implies\mu_1\le \ell-1,\\ \mu\in\cal P^\ell_\tau(\lambda)\implies \mu_\ell>1.\end{gather*}\end{lem}
\begin{pf} The first statement is clear from the definition of $\cal P^\ell_{\rm d}(\lambda)$ while for the second  note that if $\mu=\tau_\ell\nu$ for some $\nu\in\cal P^\ell_{\rm d}(\lambda)$, then   $\mu_\ell=\ell+1-\nu_1\ge 2.$ \end{pf}
\subsection{} We illustrate the recursive definition of $\cal P^\ell(\lambda)$ in a simple case using Young diagrams. Consider the case when $\lambda=(1\ge 1)\in\cal P^2$.
Then, the elements of $\cal P^2(\lambda)$ and  $\cal P^3(\lambda)$ are obtained as follows:

\vskip 24pt

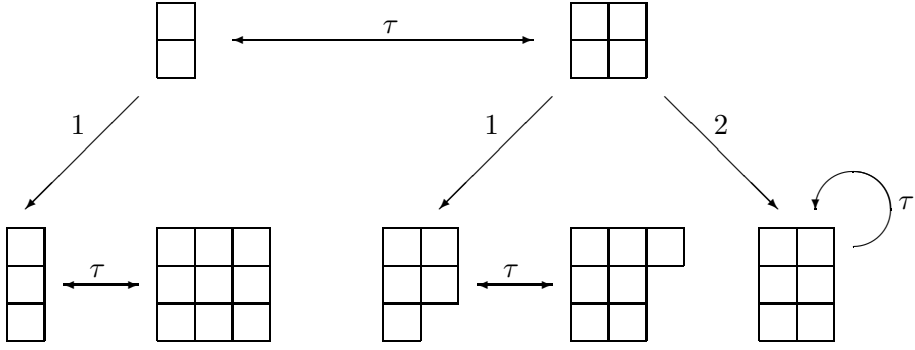
\setlength{\unitlength}{5mm}
\begin{picture}(10,10)

 \multiput(5,8)(0,1){3}{\line(1,0){1}}
 \multiput(5,8)(1,0){2}{\line(0,1){2}}
 \put(4.5,7.5){\vector(-1,-1){3}}
 \put(2.7,6.5){$1$}
 \put(7,9){\vector(1,0){8}}
 \put(15,9){\vector(-1,0){8}}
 \put(11,9.2){$\tau$}

 \multiput(16,8)(0,1){3}{\line(1,0){2}}
 \multiput(16,8)(1,0){3}{\line(0,1){2}}
 \put(15.5,7.5){\vector(-1,-1){3}}
 \put(13.7,6.5){$1$}
 \put(18.5,7.5){\vector(1,-1){3}}
 \put(19.8,6.5){$2$}

 \multiput(1,1)(0,1){4}{\line(1,0){1}}
 \multiput(1,1)(1,0){2}{\line(0,1){3}}
 \put(2.5,2.5){\vector(1,0){2}}
 \put(4.5,2.5){\vector(-1,0){2}}
 \put(3.2,2.7){$\tau$}

 \multiput(5,1)(0,1){4}{\line(1,0){3}}
 \multiput(5,1)(1,0){4}{\line(0,1){3}}

 \multiput(11,2)(0,1){3}{\line(1,0){2}}
 \multiput(11,2)(1,0){3}{\line(0,1){2}}
 \multiput(11,2)(1,0){2}{\line(0,-1){1}}
 \put(11,1){\line(1,0){1}}
 \put(13.5,2.5){\vector(1,0){2}}
 \put(15.5,2.5){\vector(-1,0){2}}
 \put(14.2,2.7){$\tau$}

 \multiput(16,1)(1,0){3}{\line(0,1){3}}
 \multiput(16,1)(0,1){2}{\line(1,0){2}}
 \multiput(16,3)(0,1){2}{\line(1,0){3}}
 \put(19,3){\line(0,1){1}}

 \multiput(21,1)(1,0){3}{\line(0,1){3}}
 \multiput(21,1)(0,1){4}{\line(1,0){2}}

 \put(23.5,4.5){\oval(2,2)[r]}
 \put(23.5,4.5){\oval(2,2)[t]}
 \put(22.5,4.5){\vector(0,-1){.01}}
 \put(24.7,4,5){$\tau$}

\end{picture}

\subsection{} For $k\in\bz_+$, set \begin{gather*} \cal P^k_{\sq}=\{\lambda\in\cal P^k:\lambda_1=k,\ \ \lambda_k=1\}.\end{gather*}
The following is the main result of this section.

\begin{thm} \begin{enumerit}
\item[(i)] Let $\ell, k\in\bn$ be such that $\ell\ge k$ and let $\lambda\in\cal P^k_{\sq}$. Then, $$\#\cal P^\ell(\lambda)=\begin{cases} c_{\ell-k+1},\ \ \ \lambda=\tau_k\lambda,\\ \ \\  2c_{\ell-k+1}, \ \ \lambda\ne \tau_k\lambda.\end{cases}$$
    \item[(ii)] Let $\lambda\in\cal P^k_{\sq}$, $\nu\in\cal P^s_{\sq}$. For all $\ell\in\bz_+$ with $\ell\ge\max(k,s)$, we have $$\cal P^\ell(\lambda)\cap\P^\ell(\nu)=\emptyset,\ \ {\rm{if}}\ \ \nu\notin\{\lambda,\tau_k(\lambda)\}.$$
        \item[(iii)] We have $$\cal P^\ell =\bigsqcup_{\{\lambda\in\cal P^k_{\sq}: \ \ell\ \ge \ k\ge\  1\}}\cal P^\ell(\lambda).$$

\end{enumerit}
\end{thm}

\begin{rem} The first part of the  theorem in particular proves that applying the algorithm $m$ times to any element of $\cal P^k_{\sq}$ for any $k\ge 1$ produces a set whose cardinality depends only on $m$ and the size of the orbit of the initial element. The remaining parts prove that the algorithm gives a partition of the set $\cup_{\ell\ge 1}\cal P^\ell$. \end{rem}

\subsection{} We note the following corollary of the theorem.
\begin{cor} For $\ell\ge 1$, we have $$\ell \cdot c_{\ell+1}= \sum_{i=1}^\ell(\ell-i+2)c_ic_{\ell-i+1}.$$
\end{cor}
\begin{pf} By the theorem, we know that $$\#\cal P^\ell=\sum_{\lambda\in\cal P^k_{\rm sq}, \ k\le \ell}c_{\ell-k+1}\#\cal P^k_{\rm sq}(\lambda).$$ The set $\cal P^\ell$ can be identified with the set of paths on the plane from $(1,0)$ to $(\ell,\ell)$ which only go right (or up) along the $x$-axis ($y$--axes) and there are exactly $\binom{2\ell-1}{\ell}$ of them. If $k=1$ then $\cal P^k_{\rm sq}$ has exactly one element. If $k>1$, then again, we can identify elements of  $\cal P^k_{\rm sq}$ with paths from $(1,1)$ to $(k,k-1)$ with only rights and ups and there are $ \binom{2k-3}{k-1}$ of them. Thus we have proven combinatorially that $$c_\ell + \sum\limits_{k=1}^{\ell-1} \binom{2\ell -2k-1}{\ell-k}c_k = \binom{2\ell-1}{\ell},$$ and if  we subtract $c_\ell = \binom{2\ell}{\ell} - \binom{2\ell}{\ell-1}$ from both sides, the right hand side simplifies and we get $\sum\limits_{k=1}^{\ell-1} \binom{2\ell -2k-1}{\ell-k}c_k = \binom{2\ell-1}{\ell-2}$and reindexing, we find
$$\sum\limits_{k=1}^{\ell} \binom{2\ell -2k+1}{\ell-k}c_k = \binom{2\ell+1}{\ell-1},$$ and after some further simple algebraic manipulation we get the corollary.

\end{pf}

\subsection{} We shall prove the theorem in Section 2.   For the rest of this section we show that our algorithm defines a tree if $\lambda=\tau_k\lambda$ and a forest with two trees if $\lambda\ne\tau_k\lambda$ and we study the propagation of the tree and forest respectively.  The first step in this is to observe that the sets $\cal P^\ell_d(\lambda)$ and $\cal P^\ell_\tau(\lambda)$ need  not be disjoint and to  identify the intersection of the two sets.
 \begin{prop}  Let $\ell\ge k\ge 1$, $\lambda\in\cal P^k_{\sq}$.
 We have \begin{eqnarray*}\cal P^\ell_{\rm d}(\lambda)\cap\cal P^\ell_\tau(\lambda)&=\{\mu\in\cal P_{\rm d}^\ell(\lambda): \ell-1\ge\mu_1\ge\mu_\ell\ge 2\}\\ &=\{\mu\in\cal P_{\tau}^\ell(\lambda): \ell-1\ge\mu_1\ge\mu_\ell\ge 2\}\\ & =\{\mu\in\cal P^\ell(\lambda): \ell-1\ge\mu_1\ge\mu_\ell\ge 2\}.\end{eqnarray*}

 \end{prop}
\begin{pf}
Notice  that $$\mu\in\cal P^\ell_{\rm d}(\lambda)\cap\cal P^\ell_\tau(\lambda)\implies \mu\setminus\{\mu_\ell\},\ \ \tau_\ell\mu\setminus\{\ell+1-\mu_1\}\in\cal P^{\ell-1}(\lambda),$$ and hence we get $$\mu_1\le\ell-1,\ \ \ell+1-\mu_\ell\le\ell-1,\ \ {\rm{i.e.}}\ \ \ell-1\ge\mu_1\ge\mu_\ell\ge 2.$$
To prove the reverse inclusion we proceed by induction on $\ell$. If $\ell\in\{k, k+1\}$ then there does not exist $\mu\in\cal P^\ell$ with $2\le\mu_\ell\le\mu_1\le\ell-1$ and hence induction begins.   For the inductive step,  we must prove that if $\mu\in\cal P^\ell_d(\lambda)$ is such that $\ell-1\ge\mu_1\ge\mu_\ell\ge 2$, then $\tau_\ell\mu\in\cal P^\ell_d(\lambda)$, i.e. that  $\tau_\ell\mu\setminus\{\ell+1-\mu_1\}\in\cal P^{\ell-1}(\lambda)$.
This is  equivalent to proving that $\tau_{\ell-1}( \tau_\ell\mu\setminus\{\ell+1-\mu_1\})\in\cal P^{\ell-1}(\lambda)$, i.e. that \begin{equation}\label{nu'} \mu'= \mu_2-1\ge\cdots\ge\mu_{\ell-1}-1\ge\mu_{\ell}-1\in\cal P^{\ell-1}(\lambda)\end{equation} and in fact we claim that   $\mu'\in\cal P^{\ell-2}_{\rm d}(\lambda).$ Consider the case when  $\mu\setminus\{\mu_\ell\}\in\cal P^{\ell-1}_{\rm d}(\lambda)$; then the induction hypothesis applies and we get $\tau_{\ell-1}(\mu\setminus\{\mu_\ell\})\in\cal P^{\ell-1}_{\rm d}(\lambda)$, i.e. that $$\nu= \ell-\mu_{\ell-1}\ge\cdots\ge\ell-\mu_2\in\cal P^{\ell-2}(\lambda),$$ and hence $$\mu'=\tau_{\ell-2}\nu=\mu_2-1\ge\cdots\ge\mu_{\ell-1}-1\in\cal P^{\ell-2}(\lambda),$$   and we are done.
Now suppose that  $\mu\setminus\{\mu_\ell\}\in\cal P^{\ell-1}_\tau(\lambda)$. This means precisely that \\
$\tau_{\ell-1}(\mu\setminus\{\mu_\ell\})\in\cal P^{\ell-1}_{\rm d}(\lambda)$ and the preceding argument repeats and proves this case. The other statements of the proposition are now clear.

 \end{pf}

 \begin{cor}\label{disjt}  For all $\ell\ge k$, we have  \begin{eqnarray*}\cal P^\ell(\lambda)&= \cal P^\ell_{\rm d}(\lambda)\ \ \bigsqcup \ \ \{\mu\in\cal P^\ell_\tau(\lambda):\mu_1=\ell\}\\&=\cal P_\tau^\ell(\lambda)\ \ \bigsqcup\ \ \{\mu\in\cal P^\ell_{\rm d}(\lambda):\mu_\ell=1\}.\end{eqnarray*} 
 \end{cor}
 \begin{pf} Suppose that $\mu\in\cal P^\ell(\lambda)$ and $\mu\notin\cal P^\ell_{\rm d}(\lambda)$. The proposition implies that we must then have either $\mu_1=\ell$ or $\mu_\ell=1$. Since $\mu\in\cal P^\ell_\tau(\lambda)$ this means by \eqref{min} that $\mu_\ell\ne 1$ and hence we have $\mu_1=\ell$ which proves the first equality of the corollary. The second follows by applying $\tau_\ell$.\end{pf}

\subsection{}   Given any $\mu= (\mu_1\ge\cdots\ge\mu_\ell)\in\cal P^\ell$, set $$\bod(\mu)=\{(\mu:j): 1\le j\le \mu_\ell\}\cup\{\tau_{\ell+1}(\mu:1)\}\subset\cal P^{\ell+1}.$$  It is easily checked that if  $\mu\ne \mu'$ then $$\bod(\mu)\cap\bod(\mu')=\emptyset.$$ For, it  is clear that $(\mu:j)=(\mu':j')$ for some $j,j'$ implies that $\mu=\mu'$. If $(\mu,j)=\tau_{\ell+1}(\mu':1)$ for some $j$, then we would have $\mu_1=\ell+1$ which is impossible since $\mu\in\cal P^\ell$. By Corollary \ref {disjt} we see that \begin{equation}\label{desd}\cal P^{\ell+1}(\lambda)=\bigsqcup_{\nu\in\cal P^\ell(\lambda)} \bod(\nu).\end{equation}

 Let $\lambda\in\cal P^k_{\sq}$ be such that $\lambda=\tau_k(\lambda)$ and define a tree $\bt_\lambda$ as follows. The set of vertices of the tree is  $$\cal P(\lambda)=\bigsqcup_{\ell\ge k}\cal P^\ell(\lambda),$$ and two vertices  $\mu,\nu\in\cal P(\lambda)$ are connected by an edge precisely if $\nu\in\bod(\mu)$ or vice-versa. The tree $\bt_\lambda$ is clearly rooted at $\lambda$ and the elements of $\cal P^\ell(\lambda)$ are those vertices with a path of length $\ell-k$ to the root.   The vertices of the tree at any given level come with a natural total order defined as follows. Suppose that we have fixed an ordering of the vertices  $\cal P^{\ell}(\lambda)$; then the order on $\cal P^{\ell+1}(\lambda)$ is as follows:
 \begin{gather*} \nu\prec \nu' \implies \mu\prec\mu'\ \ \forall\ \   \mu\in\bod(\nu)\ \  \mu'\in\bod(\nu')
 \end{gather*} and the ordering on $\bod(\nu)$ is given by $$(\nu:1)\prec(\nu: 2)\prec\cdots\prec(\nu:\nu_\ell)\prec\tau(\nu:1).$$

In the case when $\lambda\ne\tau_k\lambda$ the preceding construction gives a forest of two trees $\boF_\lambda$ rooted at $\lambda$ and $\tau_k\lambda$ respectively.

 \subsection{} The next result shows that the tree  $\bt_\lambda$  (resp. forest) is  independent of the choice of $\lambda$ and $k$ if $\lambda=\tau_k\lambda$ (resp. $\lambda\ne\tau_k\lambda$). In particular, it proves a weaker form of Theorem 1(i). \begin{prop} Let $\lambda^{(j)}\in\cal P_{\rm sq}^{k_j}$ for $j=1,2$  and assume that $k_1\ge k_2$ and also that $\tau_{k_j}\lambda^{(k_j)}=\lambda^{(k_j)}$. There exists a bijection of sets $\psi:\cal P(\lambda^{(1)})\to\cal P(\lambda^{(2)})$ such that
  \begin{gather}\nonumber\mu=\mu_1\ge\mu_2\ge\cdots\mu_\ell\in\cal P^\ell(\lambda^{(1)})\implies  \ \ \ \psi(\mu)=\mu_1'\ge\cdots\ge\mu_{\ell-k_1+k_2}'\in\cal P^{\ell-k_1+k_2}(\lambda^{(2)}), \\ \label{fl} {\rm and}\ \ \ \mu_1-\mu_1'=k_1-k_2,\ \ \mu_\ell=\mu_{\ell-k_1+k_2}'.\end{gather}
Moreover, $\psi$ induces an isomorphism $\bt_{\lambda^{(1)}}\cong \bt_{\lambda^{(2)}}.$ Analogous statements hold if $\tau_{k_j}\lambda^{(k_j)}\ne \lambda^{(k_j)}$ for $j=1,2.$
 \end{prop}
 \begin{pf} We define the map $\psi$ inductively.  Thus we set $\psi(\lambda^{(1)})=\lambda^{(2)}$ and note that the conditions of the proposition are satisfied. Assume now that we have defined $\psi:\cal P^\ell(\lambda^{(1)})\to\cal P^{\ell-k_1+k_2}(\lambda^{(2)})$ suitably  and for $\nu\in\cal P^{\ell}(\lambda^{(1)})$ define $\psi:\bod(\nu)\to\bod(\psi(\nu))$ by $$\psi(\nu:j)=(\psi(\nu):j),\ \ 1\le j\le \nu_\ell,\ \ \psi(\tau_{\ell+1}(\nu:1))=\tau_{\ell-k_1+k_2+1}(\psi(\nu:1)).$$ Since $\nu_\ell=\psi(\nu)_{\ell-k_1+k_2}$ it follows by using equation \eqref{desd} that $\psi$ extends to a bijection from $\cal P^{\ell+1}(\lambda^{(1)})\to\cal P^{\ell-k_1+k_2+1}(\lambda^{(2)})$ and it is easily checked that the extension satisfies \eqref{fl}. It is now immediate that the corresponding trees are isomorphic.
 \end{pf}
\subsection{} We shall now see that the propagation of the tree $\bt_\lambda$ is independent of $\lambda$. For this, we define another rooted tree $\bt$ as follows. The vertices of the tree will be labeled either by a pair of integers or a single integer.  The root $v_0$ of the tree will be labeled   $(2,2)$. The root has two descendants and they  have labels $(2,3)$ and $(3)$, respectively. The tree propagates as follows.  A vertex labeled $(i,j)$ will have $i$ descendants. The first $(i-1)$ descendants   will be labeled $(k,j+1)$, for $2\leq k\leq i$, and the last descendant will be labeled$(j+1)$. A vertex labeled $(r)$ will have $r$ descendants,, the first $(r-1)$ are labeled $(k,3)$ for $2\leq k\leq r$, and the last one is labeled  $(3)$. Note that the label $(2,2)$ never appears again and is uniquely associated to the root. The following picture shows the labeling at the first few levels.

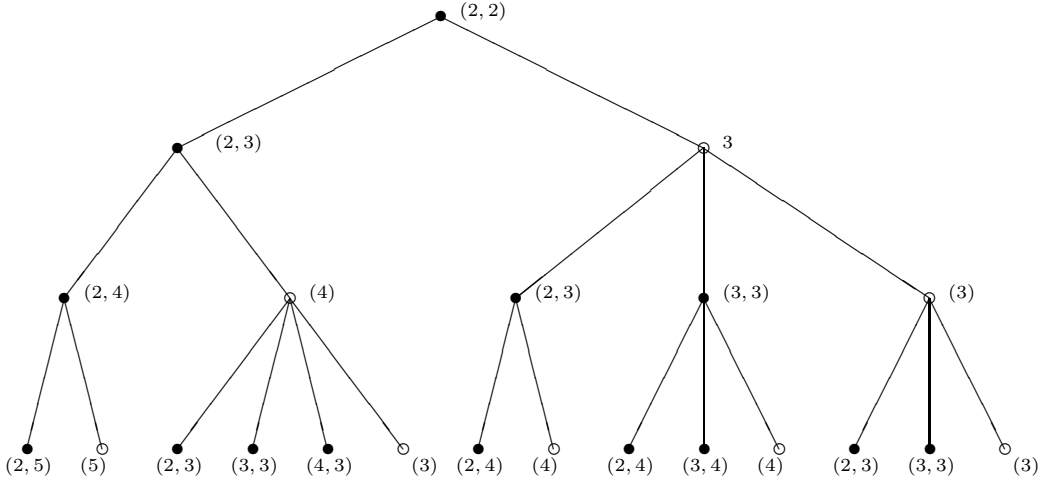
\setlength{\unitlength}{5mm}
\begin{picture}(28,17)
 \linethickness{.2mm}

 \put(12,16.5){\circle*{.3}}
 \put(12,16.5){\line(-2,-1){7}}
 \put(12,16.5){\line(2,-1){7}}
 \put(12.5,16.5){\tiny $(2,2)$}

 \put(5,13){\circle*{.3}}
 \put(19,13){\circle{.3}}
 \put(5,13){\line(-3,-4){3}}
 \put(5,13){\line(3,-4){3}}
 \put(19,13){\line(-5,-4){5}}
 \put(19,13){\line(0,-1){4}}
 \put(19,13){\line(3,-2){6}}
 \put(6,13){\tiny $(2,3)$}
 \put(19.5,13){\tiny$3$}

 \put(2,9){\circle*{.3}}
 \put(8,9){\circle{.3}}
 \put(14,9){\circle*{.3}}
 \put(19,9){\circle*{.3}}
 \put(25,9){\circle{.3}}
 \put(2,9){\line(-1,-4){1}}
 \put(2,9){\line(1,-4){1}}
 \put(8,9){\line(-3,-4){3}}
 \put(8,9){\line(-1,-4){1}}
 \put(8,9){\line(1,-4){1}}
 \put(8,9){\line(3,-4){3}}
 \put(14,9){\line(-1,-4){1}}
 \put(14,9){\line(1,-4){1}}
 \put(19,9){\line(-1,-2){2}}
 \put(19,9){\line(0,-1){4}}
 \put(19,9){\line(1,-2){2}}
 \put(25,9){\line(-1,-2){2}}
 \put(25,9){\line(0,-1){4}}
 \put(25,9){\line(1,-2){2}}
 \put(2.5,9){\tiny $(2,4)$}
 \put(8.5,9){\tiny $(4)$}
 \put(14.5,9){\tiny $(2,3)$}
 \put(19.5,9){\tiny $(3,3)$}
 \put(25.5,9){\tiny $(3)$}
,
 \put(.8,5){\circle*{.3}}
 \put(2.8,5){\circle{.3}}
 \put(4.8,5){\circle*{.3}}
 \put(6.8,5){\circle*{.3}}
 \put(8.8,5){\circle*{.3}}
 \put(10.8,5){\circle{.3}}
 \put(12.8,5){\circle*{.3}}
 \put(14.8,5){\circle{.3}}
 \put(16.8,5){\circle*{.3}}
 \put(18.8,5){\circle*{.3}}
 \put(20.8,5){\circle{.3}}
 \put(22.8,5){\circle*{.3}}
 \put(24.8,5){\circle*{.3}}
 \put(26.8,5){\circle{.3}}
 \put(0.2,4.4){\tiny $(2,5)$}
 \put(2.2,4.4){\tiny $(5)$}
 \put(4.2,4.4){\tiny $(2,3)$}
 \put(6.2,4.4){\tiny $(3,3)$}
 \put(8.2,4.4){\tiny $(4,3)$}
 \put(11,4.4){\tiny $(3)$}
 \put(12.2,4.4){\tiny $(2,4)$}
 \put(14.2,4.4){\tiny $(4)$}
 \put(16.2,4.4){\tiny $(2,4)$}
 \put(18.2,4.4){\tiny $(3,4)$}
 \put(20.2,4.4){\tiny $(4)$}
 \put(22.2,4.4){\tiny $(2,3)$}
 \put(24.2,4.4){\tiny $(3,3)$}
 \put(27,4.4){\tiny $(3)$}
\end{picture}

\begin{prop} For $\lambda\in\cal P^k_{\rm sq}$, we have an isomorphism of trees $\psi: \bt_\lambda\cong \bt$ such that  $\psi(\lambda)=v_0$ and if $\nu\in\cal P^\ell(\lambda)$ then the vertex $\psi(\nu)$ has label $(\nu_\ell+1,\ell+2-\nu_1)$, if $\nu\in\cal P_{\rm d}^\ell(\lambda)$, and label $(\nu_{\ell+1})$ if $ \nu\in\cal P^\ell_\tau(\lambda).$
\end{prop}
\begin{proof}
We define the isomorphism inductively and note that  mapping the root $\lambda$ of $\bt_\lambda$ to the root $v_0$ of $\bt$ gives the desired labels. Let $\bt^s$, $\bt^s_\lambda$ be the subtree of $\bt$ and $\bt_\lambda$ respectively, consisting of the first $s$--propagations of the root. Assume that we have defined the  isomorphism $\psi: (\bt_\lambda)_{s-1}\to \bt_{s-1}$ with the desired properties.  By \eqref{desd} it suffices to show that we can extend $\psi$ to a map from $\bod(\nu)\to\bt$ for all $\nu\in\cal P^{s+k}(\lambda)$. Suppose first that $\nu\in\cal P^{s+k}_{\rm d}(\lambda)$, in which case $\psi(\nu)$ has label $(\nu_{s+k}+1,s+k+2-\nu_1)$. This means that the vertex $\psi(\nu)$ has $\nu_{s+k}$--descendants with labels $(j,s+k+3-\nu_1)$, $2\le j\le \nu_{s+k}+1$ and one descendant with label $(s+k+3-\nu_1)$. Thus if $(\nu:j)\in\bod(\nu)$, we let $\psi((\nu:j))$ be the vertex which is the descendant of $\psi(\nu)$ with label $(j+1,s+k+3-\nu_1)$ and $\psi$ maps  $\tau_{s+k+1}(\nu:1)$ to the vertex with label $(s+k+3-\nu_1)$.

Similarly, if $\nu\in\cal P^{s+k}_\tau(\lambda)$ then $\psi(\nu)$ has label $(s+k+1-\nu_1)$ and hence the vertex $\psi(\nu)$ has  $s+k-\nu_1$--descendants with labels $\left\{(k,3): 2\leq k\leq s+k+1-\nu_1\right\}$ and  one descendant with label $\left\{(3)\right\}.$
This time, we  assign to a descendant $(\nu:m)$ the vertex labeled $(m+1,3)$ and to the descendant  $\tau_{s+k+1}(\nu:1)$ to the vertex labeled $(3)$.
This establishes a bijection between the vertices of $\bt_\lambda$ and the vertices of $\bt$, which by construction is now an isomorphism of trees.
\end{proof}

\vfill\eject

  \section{Proof of Theorem 1} \subsection{}  For $r\in\bz_+$,  set \begin{gather*}\cal P^\ell(\lambda, r)=\{\mu\in\cal P^\ell(\lambda):\mu_\ell\ge r\}, \\ \\
  e_{\ell,r}(\lambda)=\# \cal P^\ell(\lambda, r).\ \ \end{gather*}  The subsets $\cal P^\ell_{\rm d}(\lambda,r)$, $\cal P^\ell_\tau(\lambda,r)$ are defined in the obvious way. Note that by applying $\tau_\ell$, we get \begin{equation}\label{inv} e_{\ell,r}(\lambda)=\#\{\mu\in\cal P^\ell(\lambda):\mu_1\le \ell+1-r\}.\end{equation} To prove part (i) of the  Theorem we must prove that $e_{\ell,1}$ is $c_{\ell-k+1}$. We do this by showing that the $e_{\ell,r}$ satisfy a  suitable recurrence relation and by determining the initial conditions; this is the content of the next proposition.

 \begin{prop} \label{recurs}Let $r,\ell\in\bn$ and assume $\ell\ge k$. We have \begin{equation}\label{recurs1} e_{\ell,r}(\lambda)=\sum_{s\ge r-1}e_{\ell-1,s
 }(\lambda)= e_{\ell-1,r-1}(\lambda)+e_{\ell,r+1}(\lambda).\end{equation}Moreover, \begin{gather}\label{ind4} e_{\ell,r}(\lambda)=0, \ \ r>\ell-k+1,\\ \label{ind5} e_{\ell, \ell-k+1}(\lambda)=\#\cal P^k(\lambda).\end{gather}

 \end{prop}
\subsection{}  Before proving the proposition, we deduce part (i) of the theorem. It is clear that the system of recurrence relations with the initial conditions given in the proposition have a unique solution. It is also well--known \cite{Aigner} and is a simple matter to check that if we set $$e_{\ell, r}=\begin{cases} b_{\ell-k-r+1,r},\ \ {\rm if}\ \ \lambda=\tau_k\lambda,\\
2b_{\ell-k-r+1,r},\ \ {\rm if}\ \ \lambda\ne \tau_k\lambda\end{cases}$$
 then the recurrence relation and the initial conditions are satisfied. Since $$e_{\ell,1}(\lambda)=\#\cal P^\ell(\lambda) =\begin{cases}c_{\ell-k+1}	 \ \ {\rm if}\ \lambda=\tau(\lambda)\\
2c_{\ell-k+1}	\ \ {\rm else}\end{cases}$$ part (i) is proved.
 \subsection{} To prove \eqref{recurs1}  it is clear  that $$\cal P^\ell_{\rm d}(\lambda,r)=\bigsqcup_{s\ge r}\{(\mu:s):\mu\in\cal P^{\ell-1}(\lambda,s)\},$$ and hence $$\#\cal P^\ell_{\rm d}(\lambda,r)= \sum_{s\ge r}e_{\ell-1,s}(\lambda).$$By Corollary \ref{disjt} we may write $$\cal P^\ell(\lambda,r)=\cal P^\ell_{\rm d}(\lambda,r)\bigsqcup\{\mu\in\cal P^\ell_\tau(\lambda,r): \mu_1=\ell\}.$$
   We have a bijection of sets$$\{\mu\in\cal P^\ell_\tau(\lambda,r):  \mu_1=\ell\}\to  \{\nu\in\cal P^{\ell-1}(\lambda):\nu_1\le \ell+1-r\},$$ given by $$\mu\to\mu\setminus\{1\}\ \ \ \ \ \nu\to \tau_\ell(\nu:1),$$ hence  by using \eqref{inv} we see that $$\#\{\mu\in\cal P^\ell_\tau(\lambda,r):\mu_\ell=\ell\}=\#\{\nu\in\cal P^{\ell-1}(\lambda):\nu_1\le \ell+1-r\}=e_{\ell-1,r-1},$$ which proves \eqref{recurs1}.

\subsection{} The initial conditions \eqref{ind4} and \eqref{ind5} are clearly immediate consequences of the following.
\begin{lem}
\begin{gather}\label{ind1}\mu\in\cal P^\ell(\lambda)\implies \mu_\ell\le \ell-k+1,\ \   \ \  \mu_1\ge k,\\
\label{ind2} \{\mu\in\cal P^\ell(\lambda):\mu_\ell=\ell-k+1\}\subset\{\mu\in\cal P^\ell(\lambda):\mu_1=\ell\},\\
\label{ind3} \# \{\mu\in\cal P^\ell(\lambda):\mu_\ell=\ell-k+1\} =\#\cal P^k(\lambda).\end{gather}\end{lem}
\begin{pf} To prove \eqref{ind1} we proceed by induction on $\ell-k$. If $\ell=k$, then the result holds since $\lambda\in\cal P_{\sq}^k$ and by the definition of $\cal P^k(\lambda)$. Assume we have proved \eqref{ind1} for $\ell-k<s$ and let $\mu\in\cal P^{k+s}(\lambda)$.
If $\mu\in\cal P_{\rm d}^{k+s}(\lambda)$ (resp. $\mu\in\cal P^{k+s}_\tau(\lambda)$), then we have $$\mu_{k+s}\le \mu_{k+s-1}\le s\ \  ({\rm resp.}\ \ k+s+1-\mu_{k+s}\ge k,\ \ {\rm i.e,} \ \ \mu_{k+s}\le s+1) ,$$  and the inductive step is proved.

To prove \eqref{ind2} notice that it is obviously true if $\ell=k$. If $\ell>k$ and
 $$\mu_{\ell}=\ell-k +1\ge 2,\ \ \ \ \mu_1<\ell,$$ then Proposition \ref{disjt} applies and we get  $\mu\in\cal P_{\rm d}^{\ell}(\lambda)$. Applying \eqref{ind1} to $\mu\setminus\{\mu_\ell\}$ gives  $\mu_\ell\le \mu_{\ell-1}\le \ell-k$ which contradicts our assumption.

Suppose that $\mu\in\cal P^\ell(\lambda)$ is such that $\mu_\ell=\ell-k+1$. Using \eqref{ind2} we see that $\mu_1=\ell$. In particular, $\mu\notin\cal P_{\rm d}(\lambda)$, forcing $$\tau_\ell\mu\setminus\{1\}=(k\ge\cdots\ge \ell+1-\mu_2)\in\cal P^{\ell-1}(\lambda),$$  and hence we get $$\tau_{\ell-1}(\tau_\ell\mu\setminus\{1\})=(\mu_2-1\ge\cdots\ge\mu_\ell-1=\ell-k)\in\cal P^{\ell-1}(\lambda).$$ In other words the assignment
 $\mu\to \tau_{\ell-1}(\tau_\ell\mu\setminus\{1\})$ defines a bijection $\{\mu\in\cal P^\ell(\lambda):\mu_\ell=\ell-k+1\}\to  \{\mu\in\cal P^{\ell-1}(\lambda):\mu_\ell=\ell-k\}$ and hence \eqref{ind3} follows.

\end{pf}

\subsection{} To prove part (ii) of the theorem, assume  that  $\nu\notin\{\lambda, \tau_k\lambda\}$ and without loss of generality that $k\ge s$. To see that $\lambda\notin\cal P^{k}(\nu)$, notice that $\lambda\notin\cal P_{\rm d}^k(\nu)$ since $\lambda_1=k$ and by Lemma \ref{min} we also have $\tau_k\lambda\notin\cal P_\tau^k(\nu)$ since $\lambda_k=1$.

Suppose that $\mu\in\cal P^\ell(\lambda)\cap \cal P^\ell(\nu)$ for some $\ell>\max(s,k)$. If $2\le\mu_\ell\le\mu_1\le \ell-1$, then it follows from Proposition \ref{disjt} that $\mu\in\cal P^\ell_{\rm d}(\lambda)\cap \cal P^\ell_{\rm d}(\nu)$, i.e., that $$\mu\setminus\{\mu_\ell\}\in\cal P^{\ell-1}(\lambda)\cap\cal P^{\ell-1}(\nu),$$ which contradicts the induction hypothesis. If $\mu_1=\ell$, then $\mu\notin\cal P_{\rm d}^\ell(\lambda)\cup\cal P^\ell_{\rm d}(\nu)$ and so we must have that $$\tau_\ell\mu\in\cal P^\ell_{\rm d}(\lambda)\cap \cal P^\ell_{\rm d}(\nu).$$ But this implies that$$\tau_\ell\mu\setminus\{1\}\in\cal P^{\ell-1}_{\rm d}(\lambda)\cap \cal P^{\ell-1}_{\rm d}(\nu)$$ which is again impossible. The final case to consider is when $\mu_\ell=1$ and this is now immediate by applying $\tau_\ell$ to the previous case.

\subsection{} The following proposition proves part (iii) of the Theorem.
\begin{prop} \begin{enumerit}
\item[(i)] Let  $\lambda\in\cal P^k$, $\nu\in\cal P^s$ and $\mu\in\cal P^\ell$ with $k\le s\le \ell$. Then
$$\mu\in\cal P^\ell(\nu),\ \ \nu\in\cal P^s(\lambda)\implies \mu\in\cal P^\ell(\lambda).$$
\item[(ii)] Let $\mu\in\cal P^{\ell}$ for some $\ell\in\bn$. Then $\mu\in\cal P^\ell(\lambda)$ for some $\lambda\in\cal P^k_{\rm sq}$, $k\ge 1$. \end{enumerit}
\end{prop}

\begin{pf} We proceed by induction on $\ell-s$. If $\ell=s$, then we have $\mu=\nu$ or $\mu=\tau_s\nu$ and the statement follows since $\cal P^s(\lambda)$ is $\tau_s$--stable.  If $\ell>s$ and $\mu\in \cal P_{\rm d}^\ell(\nu)$, then by the induction hypothesis, we have $\mu\setminus\{\mu_\ell\}\in \cal P^{\ell-1}(\lambda)$ and hence by definition, $\mu\in\cal P^\ell(\lambda)$. Otherwise we have $\tau_\ell\mu\setminus\{\ell-\mu_1+1\}\in\cal P^{\ell-1}(\lambda)$ and hence $\tau_\ell\in\cal P^\ell(\lambda)$. Part (i) follows by using the fact that $\cal P^\ell(\lambda)$ is $\tau_\ell$-stable.

To prove (ii),  we  proceed by induction on $\ell$. If $\ell=1$, then $$\cal P^1=\cal P^1_{\rm sq}=\{1\},$$ and we are done.  Assume now that we have proved the result for all integers less than $\ell$ and let $\mu\in\cal P^{\ell}$.
If $\mu_1=\ell$ and $\mu_\ell=1$, there is nothing to prove.   If $\mu_1<\ell$, then set $$ s=\ell-\mu_1+1,\ \ \nu=\mu_1\ge\cdots\ge\mu_s.$$ Clearly $\mu\in\cal P^\ell(\nu)$ and since $\nu\in\cal P^s$ with $s<\ell$ we see by the induction hypothesis that $\nu\in\cal P^s(\lambda)$ for some $\lambda\in\cal P_{\rm sq}^k$. Applying part (i) of the proposition shows that $\mu\in\cal P^\ell(\lambda)$.
Finally, consider the case $\mu_1=\ell$ and  $\mu_1>1$. Then we have   $\nu =\tau_\ell\mu\in\cal P^{\ell,\ell-1}$ and we are now in the previous case and so $$\nu=\tau_\ell\mu\in\cal P^\ell(\lambda),$$ for some $\lambda\in\cal P_{\rm sq}^k$. The result again follows since  $\cal P^\ell(\lambda)$ is $\tau_\ell$--stable.
\end{pf}
We illustrate the basic idea in the proof of the preceding proposition in the following two simple examples. Take $\mu = 3\ge3\ge3\ge2\ge2\ge1$. Then

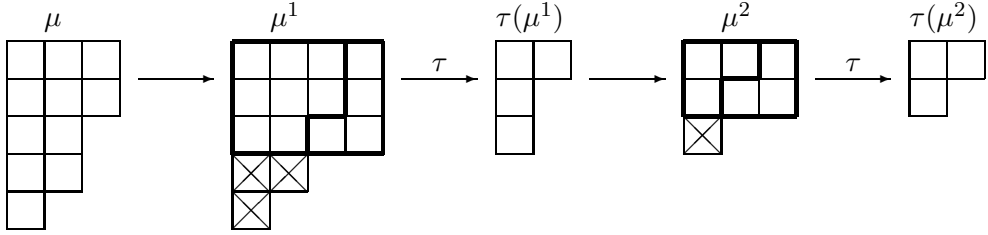
\setlength{\unitlength}{5mm}
\begin{picture}(10,6)

 \put(2,5.4){$\mu$}
 \multiput(1,1)(0,1){5}{\line(1,0){2}}
 \multiput(1,1)(1,0){3}{\line(0,1){4}}
 \multiput(1,0)(1,0){2}{\line(0,1){1}}
 \put(1,0){\line(1,0){1}}
 \multiput(3,3)(0,1){3}{\line(1,0){1}}
 \put(4,3){\line(0,1){2}}

 \put(4.5,4){\vector(1,0){2}}

 \put(8,5.4){$\mu^1$}
 \multiput(7,2)(1,0){5}{\line(0,1){3}}
 \multiput(7,2)(0,1){4}{\line(1,0){4}}
 \multiput(7,0)(1,0){2}{\line(0,1){2}}
 \multiput(7,0)(0,1){2}{\line(1,0){1}}
 \put(8,1){\line(1,0){1}}
 \put(9,1){\line(0,1){1}}
 \put(7,0){\line(1,1){2}}
 \put(7,1){\line(1,1){1}}
 \multiput(7,2)(1,0){2}{\line(1,-1){1}}
 \put(7,1){\line(1,-1){1}}

 \put(12.3,4.2){$\tau$}
 \put(11.5,4){\vector(1,0){2}}

 \put(14,5.4){$\tau (\mu^1)$}
 \multiput(14,2)(1,0){2}{\line(0,1){3}}
 \multiput(14,2)(0,1){4}{\line(1,0){1}}
 \multiput(15,4)(0,1){2}{\line(1,0){1}}
 \put(16,4){\line(0,1){1}}

 \put(16.5,4){\vector(1,0){2}}

 \put(20,5.4){$\mu^2$}
 \multiput(19,3)(1,0){4}{\line(0,1){2}}
 \multiput(19,3)(0,1){3}{\line(1,0){3}}
 \multiput(19,2)(1,0){2}{\line(0,1){1}}
 \put(19,2){\line(1,0){1}}
 \put(19,2){\line(1,1){1}}
 \put(19,3){\line(1,-1){1}}

 \put(23.3,4.2){$\tau$}
 \put(22.5,4){\vector(1,0){2}}

 \put(25,5.4){$\tau (\mu^2)$}
 \multiput(25,3)(1,0){2}{\line(0,1){2}}
 \multiput(25,3)(0,1){3}{\line(1,0){1}}
 \multiput(26,4)(0,1){2}{\line(1,0){1}}
 \put(27,4){\line(0,1){1}}

 \linethickness{.5mm}

 \multiput(7,5)(4,0){2}{\line(0,-1){3}}
 \multiput(7,5)(0,-3){2}{\line(1,0){4}}
 \put(9,2){\line(0,1){1}}
 \put(9,3){\line(1,0){1}}
 \put(10,3){\line(0,1){2}}

 \multiput(19,5)(3,0){2}{\line(0,-1){2}}
 \multiput(19,5)(0,-2){2}{\line(1,0){3}}
 \multiput(20,3)(1,1){2}{\line(0,1){1}}
 \put(20,4){\line(1,0){1}}

\end{picture}
\vskip 24pt
So we have $\mu\in\cal P^6(2\geq1)$.
\vskip 24pt

\noindent If  $\mu =(7\ge  6\ge 5\ge 3\ge 3\ge 3\ge 3\ge 3\ge 3\ge  1)$, then

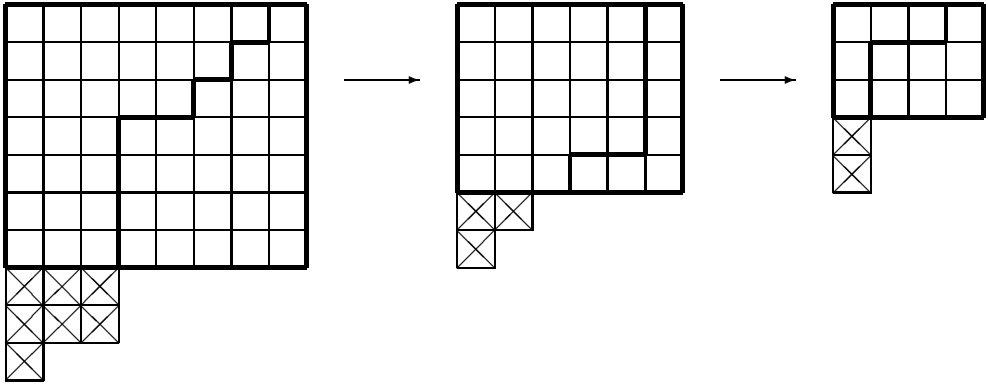
\setlength{\unitlength}{5mm}
\begin{picture}(10,10)

 \multiput(1,9)(1,0){9}{\line(0,-1){7}}
 \multiput(1,9)(0,-1){8}{\line(1,0){8}}
 \multiput(1,2)(1,0){4}{\line(0,-1){2}}
 \multiput(1,0)(0,1){2}{\line(1,0){3}}
 \multiput(1,0)(1,0){2}{\line(0,-1){1}}
 \put(1,-1){\line(1,0){1}}
 \put(1,-1){\line(1,1){3}}
 \put(1,0){\line(1,1){2}}
 \put(1,1){\line(1,1){1}}
 \put(3,0){\line(1,1){1}}
 \put(1,2){\line(1,-1){2}}
 \put(2,2){\line(1,-1){2}}
 \put(3,2){\line(1,-1){1}}
 \put(1,1){\line(1,-1){1}}
 \put(1,0){\line(1,-1){1}}

 \multiput(13,9)(1,0){7}{\line(0,-1){5}}
 \multiput(13,9)(0,-1){6}{\line(1,0){6}}
 \multiput(13,4)(1,0){3}{\line(0,-1){1}}
 \put(13,3){\line(1,0){2}}
 \multiput(13,3)(1,0){2}{\line(0,-1){1}}
 \put(13,2){\line(1,0){1}}
 \put(13,2){\line(1,1){2}}
 \put(13,3){\line(1,1){1}}
 \put(13,3){\line(1,-1){1}}
 \put(13,4){\line(1,-1){1}}
 \put(14,4){\line(1,-1){1}}

 \multiput(23,9)(1,0){5}{\line(0,-1){3}}
 \multiput(23,9)(0,-1){4}{\line(1,0){4}}
 \multiput(23,6)(1,0){2}{\line(0,-1){2}}
 \multiput(23,5)(0,-1){2}{\line(1,0){1}}
 \multiput(23,5)(0,-1){2}{\line(1,1){1}}
 \multiput(23,6)(0,-1){2}{\line(1,-1){1}}

 \put(10,7){\vector(1,0){2}}
 \put(20,7){\vector(1,0){2}}

 \linethickness{.5mm}

 \multiput(1,9)(8,0){2}{\line(0,-1){7}}
 \multiput(1,9)(0,-7){2}{\line(1,0){8}}
 \put(4,2){\line(0,1){4}}
 \put(4,6){\line(1,0){2}}
 \put(6,6){\line(0,1){1}}
 \put(6,7){\line(1,0){1}}
 \put(7,7){\line(0,1){1}}
 \put(7,8){\line(1,0){1}}
 \put(8,8){\line(0,1){1}}

 \multiput(13,9)(6,0){2}{\line(0,-1){5}}
 \multiput(13,9)(0,-5){2}{\line(1,0){6}}
 \put(16,4){\line(0,1){1}}
 \put(16,5){\line(1,0){2}}
 \put(18,5){\line(0,1){4}}

 \multiput(23,9)(4,0){2}{\line(0,-1){3}}
 \multiput(23,9)(0,-3){2}{\line(1,0){4}}
 \put(24,6){\line(0,1){2}}
 \put(24,8){\line(1,0){2}}
 \put(26,8){\line(0,1){1}}

\end{picture}
\vskip 24pt
So we have $\mu\in \cal P^{10}(3\geq 1\geq 1)$.
\vskip 24pt

\section{From the Catalan numbers to the Ballot numbers} In this section, we generalize the first part of Theorem 1. Namely, given $m\in\bn$, we modify the algorithm defined in Section 1 so that if we start with  a suitable set of $m$ elements, then applying the algorithm $\ell$ times gives a set of cardinality equal to the ballot number  $b_{\ell,m}$.  We use the binomial identity \begin{equation}\label{ballotr} \binom{r}{s}=\binom{r-1}{s}+\binom{r-1}{s-1},$$ freely and without comment  throughout the rest of the section. Note that in particular, this gives $$b_{\ell,m}=b_{\ell,m-1}+b_{\ell-1,m+1}.\end{equation}
\subsection{} Fix $m\in\bn$,   and let $$\Omega_m=\{\{j\}: 1\le j\le m\}\subset \cal P^{1,m}.$$ We generalize the definition of the sets $\cal P^\ell(\lambda)$ given in Section 1 as follows. Define subsets $\cal P^{\ell}(\Omega_m)$ by,\begin{gather*}\cal P_{\rm d}^1(\Omega_m)=\Omega_m=\cal P^1_{\tau}(\Omega_m)=\tau_m\Omega_m,\\ \cal P^\ell_{\rm d}(\Omega_m)=\{(\mu:j)\in\cal P^{\ell,\ell+m-1}: 1\le j\le \mu_{\ell},\ \ \mu\in\cal P^{\ell-1}(\Omega_m)\},\\ \cal P^\ell_\tau(\Omega_m)=\tau_{\ell+m-1}\cal P^\ell_{\rm d}(\Omega_m),\\ \cal P^\ell(\Omega_m)=\cal P^\ell_{\rm d}(\Omega_m)\cup\cal P^\ell_\tau(\Omega_m). \end{gather*}

 The main result of this section is:
 \begin{thm} For $\ell,m\in\bn$, we have  $$\#\cal P^\ell(\Omega_m)=b_{\ell,m-1}.$$
 \end{thm}
\subsection{} The proof of the theorem is very similar to the corresponding result in Section 1. An inspection of Proposition \ref{disjt} and its Corollary shows that the proof works in our more general situation and we have: \begin{prop}\label{disjt2}  For all $\ell\ge 1$, the set $\cal P^\ell(\Omega_m)$ is the disjoint union of the following sets:
 for all $\ell\ge 1$, we have  \begin{eqnarray*}\cal P^\ell(\Omega_m)&= \cal P^\ell_{\rm d}(\Omega_m)\ \ \bigsqcup \ \ \{\mu\in\cal P^\ell_\tau(\Omega_m):\mu_1=\ell\}\\&=\cal P_\tau^\ell(\Omega_m)\ \ \bigsqcup\ \ \{\mu\in\cal P^\ell_{\rm d}(\Omega_m):\mu_\ell=1\}.\end{eqnarray*}
\hfill\qedsymbol \end{prop}

\subsection{}  For $s,\ell\in\bn$, $\ell\ge 1$, set \begin{gather*}\cal P^\ell(\Omega_m,s)=\{\mu\in\cal P^\ell(\Omega_m):\mu_\ell\ge s\}, \\
  e_{\ell,s}(\Omega_m)=\# \cal P^\ell(\Omega_m, s)\end{gather*} and define $e_{\ell,0}(\Omega_m)=e_{\ell,1}(\Omega_m)$. Note that by applying $\tau_{\ell+m-1}$, we get  \begin{equation}\label{inv2} e_{\ell,s}(\Omega_m)=\#\{\mu\in\cal P^\ell(\Omega_m):\mu_1\le \ell+m-s\}.\end{equation}
Clearly, $e_{1,\ell}(\Omega_m)$ is just the cardinality of $\cal P^\ell(\Omega_m)$.
 We now determine the recurrence relation and the initial conditions satisfied by the $e_{\ell,s}$.
 \begin{prop} \label{recursm}For $\ell\ge 1$ and $s\ge 0$, we have \begin{equation}\label{elm} e_{\ell,s}(\Omega_m)=\sum_{j\ge s-1}e_{\ell-1,s}(\Omega_m)= e_{\ell-1,s-1}(\Omega_m)+e_{\ell,s+1}(\Omega_m).\end{equation}  Moreover, \begin{gather}\label{ind42} e_{\ell,s}(\Omega_m)=0\ \ \ {\rm if }\  s>\ell+m-1,\\ \label{ind52} e_{\ell,\ell+m-1}(\Omega_m)=1,\\
\label{ind62} e_{1,s}(\Omega_m)=\begin{cases} m-s+1\ \ {\rm if}\ \ \ m-s+1 \geq 0,\\
0\ \  {\rm otherwise.}\end{cases}\end{gather}

 \end{prop}
 \begin{pf} It is immediate from Proposition \ref{disjt2} and  \eqref{inv2} that \eqref{elm} holds.
Equation \eqref{ind42} holds  since by definition  $\mu\in \mathcal P^\ell(\Omega_m)$ implies $\mu_1\leq \ell+r-1$.
Let $\mu\in\mathcal P^{\ell}(\Omega_m)$ be such that $\mu_\ell\geq \ell+r-1$. Then we must have,
 $\mu=(\ell+r-1\geq\ell+r-1\geq\ldots\geq\ell+r-1)$. Further, this element is $\tau_{\ell+r-1}(1\geq1\ldots\geq1)$, and we clearly have $(1\geq1\geq\ldots\geq 1)\in \mathcal P_{\rm d}^\ell(\Omega_m)$. This proves \eqref{ind52}.
 Finally, equation \eqref{ind62} is an immediate consequence of the definitions of $e_{\ell,s}(\Omega_m)$ and of $\Omega_m$.

 \end{pf}

\subsection{} It is clear that the integers $e_{\ell,s}(\Omega_m)$ are completely determined by Proposition \ref{recursm}. The proof of the theorem is completed by the following proposition which gives closed formulae for the $e_{\ell,s}(\Omega_m)$.
\begin{prop} For $\ell,m\ge 1$ and $s\ge 0$,  we have $$e_{\ell,s}(\Omega_m)=\begin{cases} \binom{m+2\ell-s-1}{\ell},\ \ \ s\ge \ell\\ \\
\sum_{j\ge 0}(-1)^j\binom{s-j}{j}b_{\ell-j,m-1},\ \ 1\le s\le\ell-1.\end{cases}$$ In particular, $e_{\ell,1}(\Omega_m)=b_{\ell,m-1}$.
\end{prop}
\begin{pf}Notice first that the numbers on the right hand side satisfy \eqref{ind42}, \eqref{ind52} and \eqref{ind62}. The proposition follows if we prove that they also satisfy \eqref{elm}. If $s\ge \ell$ (resp. $s<\ell-1)$, then we must check that \begin{gather*}\binom{m+2\ell-s-1}{\ell}=\binom{m+2\ell-s-2}{\ell-1}+\binom{m+2\ell-s-2}{\ell},\end{gather*} and \begin{gather*}\sum_{j\ge 0}(-1)^j\binom{s-j}{j}b_{\ell-j,m-1}=\sum_{j\ge 0}(-1)^j\binom{s-j-1}{j}b_{\ell-j-1,m-1}+\sum_{j\ge 0}(-1)^j\binom{s-j+1}{j}b_{\ell-j,m-1}.
\end{gather*} The first one is just the usual binomial identity,
 while for the second, observe that \begin{eqnarray*}\sum_{j\ge 0}(-1)^j\binom{s-j}{j}b_{\ell-j,m-1}-\sum_{j\ge 0}(-1)^j\binom{s-j+1}{j}b_{\ell-j,m-1}&=\sum_{j\ge 1}(-1)^{j+1}\binom{s-j}{j-1}b_{\ell-j,m-1}\\ &=\sum_{j\ge 0}(-1)^j \binom{s-j-1}{j}b_{\ell-j+1,m-1}.\end{eqnarray*} It remains to consider the case when $s=\ell-1$, i.e. we have to verify that
 $$\sum_{j\ge 0}(-1)^j\binom{\ell-1-j}{j}b_{\ell-j,m-1}=\sum_{j\ge 0}(-1)^j\binom{\ell-j-2}{j}b_{\ell-j-1,m-1} +\binom{m+\ell-1}{\ell}.$$ This amounts to proving (by replacing $j$ with $j+1$ on the right hand side and using the binomial identity again) that \begin{equation*}\sum_{j\ge 0}(-1)^j\binom{\ell-j}{j}b_{\ell-j,m-1}=\binom{m+\ell-1}{\ell}.\end{equation*}

\pagebreak

\noindent
This is probably well--known  but we isolate it as a separate Lemma and give a proof, since we were unable to find a reference in general.

\end{pf}
\subsection{}\begin{lem} For $\ell,m\ge 0$, we have \begin{equation}\label{ballotrecurs}\sum_{j\ge 0}(-1)^j\binom{\ell-j}{j}b_{\ell-j,m}=\binom{m+\ell}{\ell}.\end{equation}
\end{lem} \begin{pf} Note that if $m=0$, then this formula is known for all $\ell$, \cite{Aigner} since the $b_{\ell,0}$ are Catalan numbers. Assume now that we have proved it for all  pairs $(\ell,m')$ with  $m'<m$. To prove it for $(\ell,m)$ we proceed again by induction on $\ell$. If $\ell=0$, the equation is just $b_{0,m}=1$ which follows from the definition.
Assuming the result for $(\ell,m)$, we prove it for $(\ell+1,m)$ as follows. Consider: \begin{eqnarray*}\sum_{j\ge 0}(-1)^j\binom{\ell+1-j}{j}b_{\ell+1-j,m}-\sum_{j\ge 0}(-1)^j\binom{\ell+2-j}{j}(b_{\ell+2-j,m-1}-b_{\ell+2-j, m-2})&\\ =
\sum_{j\ge 1}(-1)^j\binom{\ell+1-j}{j}b_{\ell+1-j,m}-\sum_{j\ge 1}(-1)^j\binom{\ell+2-j}{j}(b_{\ell+2-j,m-1}-b_{\ell+2-j, m-2})&\\ =\sum_{j\ge 1}(-1)^j\binom{\ell+1-j}{j}(b_{\ell+1-j,m}-b_{\ell+2-j,m-1}+b_{\ell+2-j, m-2})\\ + \sum_{j\ge 1}(-1)^j\binom{\ell+1-j}{j-1}(b_{\ell+2-j,m-1}-b_{\ell+2-j, m-2})\\ = 0+\sum_{j\ge 0}(-1)^j\binom{\ell-j}{j}b_{\ell-j,m},\end{eqnarray*} i.e.

\[ \sum_{j\ge 0}(-1)^j\binom{\ell+1-j}{j}b_{\ell+1-j,m}\] \[= \sum_{j\ge 0}(-1)^j\binom{\ell+2-j}{j}(b_{\ell+2-j,m-1}-b_{\ell+2-j, m-2})+\sum_{j\ge 0}\binom{\ell-j}{j}b_{\ell-j,m}. \]

\noindent For the inductive step to work, we must have the result for $m=1$ as well. For this, note  that $b_{\ell+1-j,1}=b_{\ell+2-j,0}$ and $b_{\ell,-1}=0$ for all $\ell$.  Hence, we get \begin{gather*} \sum_{j\ge 0}(-1)^j\binom{\ell+1-j}{j}b_{\ell+1-j,1}=\sum_{j\ge 0}(-1)^j\binom{\ell+2-j}{j}b_{\ell+2-j,0}+\sum_{j\ge 0}\binom{\ell-j}{j}b_{\ell-j,1}\\ \\= 1+\ell\end{gather*} as required.
In the general case, the induction hypothesis applies to all the terms on the  right hand side and we get

$$\sum_{j\ge0}(-1)^j\binom{\ell+1-j}{j}b_{\ell+1-j,m}=\binom{r+\ell+1}{\ell+2}-\binom{m+\ell}{\ell+2}+\binom{m+\ell}{\ell}=\binom{m+\ell+1}{\ell+1},$$ and the proof is complete.

\end{pf}

\section{Concluding Remarks and a Conjecture.}

 \subsection{} As we mentioned in the introduction, our motivation for this paper came from the study of the representation theory of affine Lie algebras. There is a well--known relationship \cite{G}, \cite{BCP},\cite{CP} between the ring of symmetric functions in infinitely many variables and the universal enveloping algebra of the affine Lie algebra. The problem we are interested in leads us naturally through this connection to the following conjecture.

 For $r\ge 1$, let $\bc[x_1,\cdots, x_r]$ be the polynomial ring in $r$--variables, $S_r$ be the symmetric group on $r$ letters, and let $$\Lambda_r=\bc[x_1,\cdots, x_r]^{S_r}$$ be the ring of invariants under the canonical action of $S_r$ on the polynomial ring.  Given elements $a,b\in \bc[x_1,\cdots, x_r]$ and $m\ge 0$, set $$\bop_0(a,b)=1,\ \ \bop_m(a,b)=  \sum_{j=0}^{m-1}a^{m-j-1}b^j.$$ Given a partition $\mu=\mu_1\ge\cdots\ge\mu_s>0$  let ${\rm comp}(\mu)\subset \bz_+^{s}$ be the set of all  distinct elements arising from permutations of $(\mu_1,\cdots,\mu_s)\in\bz_+^s$.

 From now on, we consider the case when   $r=2\ell+m$ for some $\ell,m\ge 1$. Given $\mu\in\cal P^\ell$, we set $$\bop(\mu)=\sum_{\mu'\in{\rm comp}(\mu)}\bop_{\mu'_1} (x_1,x_2)\cdots\bop_{\mu'_\ell}(x_{2\ell-1},x_{2\ell}).$$ Let $\bm(\ell,m)$ be the $\Lambda_r$--submodule of $\bc[x_1,\cdots, x_r]$ spanned by the elements $\bop(\mu)$, $\mu\in\cal P^\ell$. We can now state our conjecture.

\noindent {\bf Conjecture} The $\Lambda_r$--module $\bm(\ell,m)$ is free with basis $$\{\bop_\mu:\mu\in\cal P^\ell(\Omega_m)\},$$ and in particular is of rank $b_{\ell,m-1}$.

 We have checked that the conjecture is true for all $m$ if $\ell=1,2$ and for $\ell=3,4$ for $m=0,1,2$.

 \subsection{} There are other natural generalizations of the algorithm. Namely, we could start with any partition $\lambda\in\cal P$ and define a subset $\cal Q(\lambda)$ by setting $$\cal Q^1(\lambda)=\{\lambda\}\cup\{\tau_{\lambda_1+\lambda_k-1}(\lambda)\},$$ and then defining $\cal Q^\ell_{\rm d}(\lambda)$ in the obvious way and $$\cal Q^\ell_\tau(\lambda)=\tau_{\lambda_1+\lambda_k+\ell-1}\cal Q^\ell_d(\lambda).$$ Computations for small values of $\ell$ and specific $\lambda$ do yield sequences of numbers found in \cite{S} for the  cardinality of the sets. The abstract result needed, however, to compute the recurrence relations in general is the analog of Corollary \ref{disjt}. The Corollary is  definitely false in this generality and it should be interesting to find  the correct statement.

\end{document}